\RequirePackage{fixltx2e}
\RequirePackage[l2tabu,orthodox]{nag}

\documentclass{article}
\usepackage{enumitem}
\setlist{nolistsep}
\usepackage{hyperref}
\hypersetup{colorlinks=true,citecolor=blue}
\usepackage{color}
\usepackage[section]{placeins}
\usepackage{psfrag, pstricks}
\usepackage{subfigure}
\usepackage{amsmath,amsthm,amsfonts,amssymb}
\usepackage{microtype}
\usepackage[utf8]{inputenc}
\usepackage{graphicx, psfrag}
\usepackage{verbatim}
\usepackage{color}

\newcommand{\R}{{\mathrm I}\kern-0.18em{\mathrm R}}
\newcommand{\h}{{\mathrm I}\kern-0.18em{\mathrm H}}
\newcommand{\K}{{\mathrm I}\kern-0.18em{\mathrm K}}
\newcommand{\p}{{\mathrm I}\kern-0.18em{\mathrm P}}
\newcommand{\E}{{\mathrm I}\kern-0.18em{\mathrm E}}
\newcommand{\1}{{\mathrm 1}\kern-0.24em{\mathrm I}}
\newcommand{\N}{{\mathrm I}\kern-0.18em{\mathrm N}}

\newcommand{\ktau}{\boldsymbol \tau}
\newcommand{\krho}{\boldsymbol \rho}
\newcommand{\diag}{\mathrm{diag}}

\newcommand{\cN}{\mathcal{N}}

\newcommand{\bA}{\mathbf{A}}

\newcommand{\bM}{\mathbf{M}}

\newcommand{\bW}{\mathbf{W}}
\newcommand{\bX}{\mathbf{X}}

\newcommand{\sD}{\mathsf{D}}
\DeclareMathOperator{\Var}{Var}
\DeclareMathOperator{\Tr}{Tr}
\DeclareMathOperator{\card}{card}
\DeclareMathOperator{\sign}{sign}

\newtheorem{theorem}{Theorem}

\newtheorem{proposition}[theorem]{Proposition}

\newtheorem{lemma}[theorem]{Lemma}

\title{Mar\v{c}enko-Pastur Law for Kendall's Tau}

\author{Afonso S.~ Bandeira\thanks{NYU, Courant Institute of Mathematical Sciences: \texttt{bandeira@cims.nyu.edu}}, Asad Lodhia\thanks{MIT, Department of Mathematics: \texttt{lodhia@math.mit.edu}}, and Philippe Rigollet\thanks{MIT, Department of Mathematics: \texttt{rigollet@math.mit.edu}}}

\begin{document}
\maketitle
\begin{abstract}
We prove that Kendall's Rank correlation matrix converges to the Mar\v{c}enko
Pastur law, under the assumption that observations are i.i.d random vectors $X_1, \ldots, X_n$
with  components that are independent and  absolutely continuous with respect to the Lebesgue
measure. This is the first result on the empirical spectral distribution of a multivariate $U$-statistic.
\end{abstract}
\section{Introduction}
Estimating the association between two random variables $X,Y \in \R$ is a central statistical problem. As such many methods have been proposed, most notably Pearson's correlation coefficient. While this measure of association is well suited to the Gaussian case, it may be inaccurate in other cases. This observation has led statisticians to consider other measures of associations such as Spearman's $\rho$ and Kendall's $\tau$ that can be proved to be more robust to heavy-tailed distributions (see, e.g., \cite{LiuHanYua12}). 
In a multivariate setting, covariance and correlation matrices are preponderant tools to understand the interaction between variables. They are also used as building blocks for more sophisticated statistical questions such as principal component analysis or graphical models. 

The past decade has witnessed an unprecedented and fertile interaction between random matrix theory and high-dimensional statistics (see~\cite{PauAue14} for a recent survey). Indeed, in  high-dimensional settings, traditional asymptotics where the sample size tends to infinity fail to capture a delicate interaction between sample size and dimension and random matrix theory has allowed statisticians and practitioners alike to gain valuable insight on a variety of multivariate problems.

The terminology ``Wishart matrices" is often, though sometimes abusively, used to refer to $p \times p$ random matrices of the form $\bX^\top \bX/n$, where $\bX$ is an $n \times p$  random matrix with independent rows (throughout this paper we restrict our attention to real random matrices). The simplest example arises where $\bX$ has i.i.d standard Gaussian entries but the main characteristics are shared by a much wider class of random matrices. This universality phenomenon manifests itself in various aspects of the limit distribution, and in particular in the limiting behavior of the empirical spectral distribution of the matrix. Let $\bW=\bX^\top \bX/n$ be a $p \times p$ Wishart matrix and denote by $\lambda_1, \ldots, \lambda_p$ its eigenvalues; then the empirical spectral distribution $\hat \mu_p$ of $\bW$ is the distribution on $\R$ defined as the following mixture of Dirac point masses at the $\lambda_j$s:
\[
\hat \mu_p=\frac{1}{p} \sum_{k=1}^p \delta_{\lambda_k}\,.
\]
Assuming that the entries of $\bX$ are independent, centered and of unit variance, it can be shown that $\mu_p$ converges weakly to the Mar\v{c}enko-Pastur distribution under weak moment conditions (see \cite{ErdKnoYau12} for the weakest condition).

While this development alone has led to important statistical advances, it fails to capture more refined notions of correlations, notably more robust ones involving ranks and therefore dependent observations. A first step in this direction was made by \cite{YinKri86}, where  the matrix $\bX$ is assumed to have independent rows with isotropic distribution. More recently, this result was extended in~\cite{BaiZho08, ORo12} and covers for example the case of Spearman's $\krho$ matrix that is based on ranks, which is also a Wishart matrix of the form $\bX^\top \bX/n$. 

The main contribution of this paper is to derive the limiting distribution of Kendall's $\ktau$ matrix, a cousin of Spearman's $\krho$ matrix but which is \emph{not} of the Wishart type but rather a matrix whose entries are $U$-statistics. Kendall's $\ktau$ matrix is a very popular surrogate for correlation matrices but an understanding the fluctuations of its eigenvalues is still missing. Interestingly, Mar\v{c}enko-Pastur results have been used as heuristics, without justification, precisely for Kendall's $\ktau$ in the context of certain financial applications~\cite{CreCreLav15}.

As it turns out, the limiting distribution of $\hat \mu_p$ is not exactly Mar\v{c}enko-Pastur, but rather an affine transformation of it. Our main theorem below gives the precise form of this transformation.
\begin{theorem}
\label{thm:main}
Let $X_1, \ldots, X_n$, be $n$ independent
random vectors in $\R^p$ whose components $X_i(k)$ are independent
random variables that have a density with respect to the Lebesgue 
measure on $\R$. Then as $n \to \infty$ and $\frac{p}{n} \to \gamma >0$ the empirical 
spectral distribution of $\ktau$ converges in probability to 
\[
\frac{2}{3}Y + \frac{1}{3}\,,
\]
where $Y$ is distributed according to the standard Mar\v{c}enko-Pastur law with parameter $\gamma$ 
(see Theorem \ref{thm:mplaw}  for the appropriate definition).
\end{theorem}

Figure~\ref{FIG:mp} illustrates numerically the result of Theorem~\ref{thm:main}.

\begin{figure}[h]
\centering{
\includegraphics[width = 0.8\textwidth]{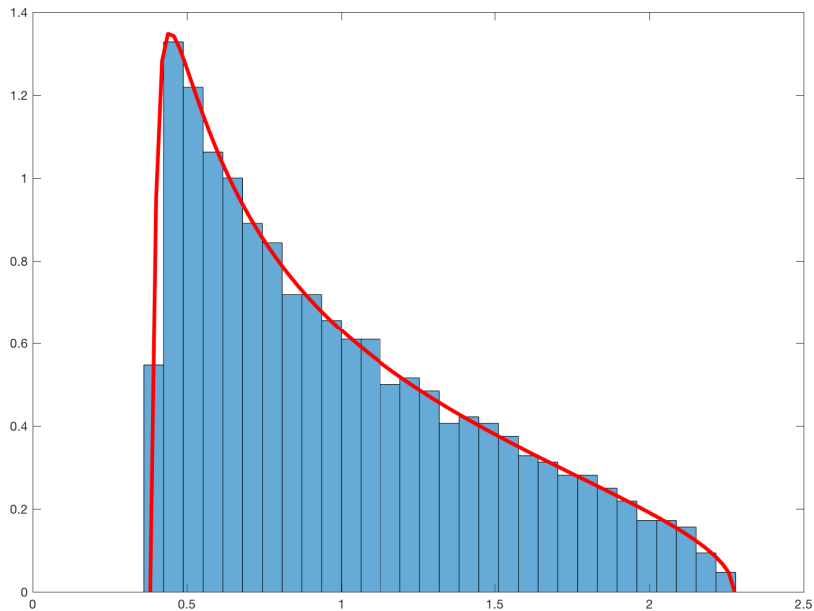}
\caption{Histogram of the eigenvalues the Kendall $\ktau$ matrix for $n=2,000, p=1,000$. The superimposed red line is the probability density function $\frac{2}{3}Y + \frac{1}{3}$, where $Y$ is distributed according to the standard Mar\v{c}enko-Pastur law with parameter $\gamma=1/2$.}
\label{FIG:mp}}
\end{figure}

\medskip \textsc{Notation:} For any integer $k \ge 1$ we write $[k]=\{1, \ldots, k\}$. We denote by $I_p$ the identity matrix of $\R^p$. For a vector $x \in \R^p$, we denote by $x(j)$ it's $j$th coordinate. For any $p \times p$ matrix $M$, we denote by $\diag(M)$ the $p \times p$ diagonal matrix with the same diagonal elements as $M$ and any real number $r$, we define  $\sD_r(M)=M-\diag(M)+rI_p$. In other words, the operator $\sD_r$ replaces each diagonal element of a matrix by the value $r$.
We denote $\sign$ the sign function with convention that $\sign(0)=1$. We define the Frobenius norm of a $p \times p$ matrix $H$ as  $\|H\|_{\mathrm{F}}^2 := \Tr(H^\top H)$. Finally, we define $\mathsf{Unif}([a,b])$ to be the uniform distribution on the interval $[a,b]$.

\section{Kendall's Tau}
\label{SEC:kendall}

The (univariate) Kendall $\tau$ statistic~\cite{Ess24,Lin25, Lin29,Ken38} is defined as follows.  Let $(Y_1,Z_1), \ldots, (Y_n,Z_n)$ be $n$ independent samples of a pair $(Y,Z) \in \R\times \R$ of real-valued random variables. Then the (empirical) Kendall $\tau$ between $Y$ and $Z$ is defined as
\[
\tau(Y,Z)=\frac{1}{{\binom{n}{2}}}\sum_{1\le i<j\le n} \sign(Y_i-Y_j)\cdot\sign(Z_i-Z_j)\,.
\]
The statistic $\tau$ takes values in $[-1,1]$ and it is not hard to see that it can be expressed as
\[
\tau= \frac{1}{\binom{n}{2}}\left(\#\left\{\text{concordant pairs}\right\}  -   \#\left\{\text{discordant pairs}\right\}\right)\,,
\]
Where a pair $(i,j)$ is said to be \emph{concordant} if $Y_i-Y_j$ and $Z_i-Z_j$ have the same sign and \emph{discordant} otherwise.

It is known  that the Kendall $\tau$ statistic is asymptotically Gaussian (see, e.g.,~\cite{Ken38}). Specifically, if $Y$ and $Z$ are independent, then as $n \to \infty$, 
\begin{equation}
\label{EQ:clt_tau}
\sqrt{n} \tau(Y,Z) \rightsquigarrow \cN\left(0,\frac{4}{9}\right) \,.
\end{equation}
This property has been central to construct independence tests between two random variables $X$ and $Y$ (see, e.g,~\cite{KenGib90}).

Kendall's $\tau$ stastistic can be extended to the multivariate case. Let $X_1, \ldots, X_n$, be $n$ independent copies of a random vector $X\in \R^p$, with independent coordinates $X(1), \ldots, X(p)$. The (empirical) \textbf{Kendall} $\ktau$ matrix of $X$ is defined to be the $p \times p$ matrix whose entries $\ktau_{kl}$ are given by
\begin{equation}
\label{eqn:ktau} 
\ktau_{kl} := \tau(X(k),X(l))=\frac{1}{\binom{n}{2}}\sum_{1\le i<j\le n} \sign\big(X_i(k)-X_j(k)\big)\cdot \sign\big(X_i(l)-X_j(l)\big)\quad  1\le k, l \le p\,.
\end{equation}
Note that the $\ktau$ can be written as the sum of $\binom{n}{2}$ rank-one random matrices:
\begin{equation}
\label{EQ:defTau_outer}
\ktau=\frac{1}{\binom{n}{2}}\sum_{1\le i<j\le n} \sign\big(X_i-X_j\big)\otimes \sign\big(X_i-X_j\big)\,,
\end{equation}
where the $\sign$ function is taken entrywise.

It is easy to see that  $\ktau_{ii} = 1$ for all $i$. Together with~\eqref{EQ:clt_tau}, it implies that the matrix 
\[
\bar \ktau=\frac{3}{2}\ktau-\frac{1}{2}I_p
\] 
is such that $\E[\sqrt{n}\bar \ktau]\to I_p$ and   $\Var[\bar \ktau_{ij}]\to \1\{i \neq j\}$ as $n \to \infty$. This suggests that if the empirical spectral distribution of $\bar \ktau$ converges to a Mar\v{c}enko-Pastur distribution, it should be a \emph{standard} Mar\v{c}enko-Pastur distribution. This heuristic argument supports the affine transformation arising in Theorem~\ref{thm:main}. However, the matrix $\ktau$ is \emph{not}  Wishart and the Mar\v{c}enko-Pastur limit distribution does not follow from standard arguments. Nevertheless,  Kendall's $\tau$ is a $U$-statistic which are known to satisfy the weakest form of universality, namely a Central Limit Theorem under general conditions~\cite{Hoe48, PenGin99}. In this paper, we show that in the case of the Kendall $\ktau$ matrix, this universality phenomenon extends to the empirical spectral distribution.

\section{Proof of Theorem~\ref{thm:main}}

For any pair $(i,j)$ such that $1\le i,j \le n$ and $i \neq j$, let $\bA_{(i,j)}$ be the vector
\[
\bA_{(i,j)} := \sign\big(X_i-X_j\big)\,,
\]
and recall from~\eqref{EQ:defTau_outer} that 
\[
\ktau=\frac{1}{\binom{n}{2}}\sum_{1\le i<j\le n}\bA_{(i,j)}\otimes \bA_{(i,j)}\,.
\]

Akin to most asymptotic results on $U$-statistics, we utilize a variant of Hoeffding's (a.k.a. Efron-Stein, a.k.a ANOVA) decomposition~\cite{Hoe48}:
\begin{equation}
\label{EQ:ANOVA}
\bA_{(i,j)}=\bar \bA_{(i,j)}+ \bar \bA_{(i,\cdot)}+ \bar \bA_{(\cdot,j)}
\end{equation}
where
\[
\bar \bA_{(i,\cdot)}:=\E\big[\bA_{(i,j)}\big| X_i\big]\,, \quad \bar \bA_{(\cdot,j)}:=\E\big[\bA_{(i,j)}\big| X_j\big]\quad \hbox{and}\quad 
\bar \bA_{(i,j)}:= \bA_{(i,j)} - \bar \bA_{(\cdot,j)} - \bar \bA_{(i,\cdot)}\,.
\]
It is easy to check that each of the vectors in the right-hand side of~\eqref{EQ:ANOVA} are centered and are orthogonal to each other with respect to the inner product $\E [v^\top w]$ where $v,w \in \R^p$. These random vectors can be expressed conveniently thanks to the following Lemma.
\begin{lemma}
\label{lem:Abar}
For $k \in [p]$, let $F_{k}$ denote the cumulative distribution function of $X(k)$.
Fix $i \in [n]$ and let $U_i \in \R^p$ is a random vector with $k$th coordinate given by $U_i(k)= 2F_{k}(X_i(k)) - 1\sim \mathsf{Unif}([-1,1])$. Then
\[
\bar \bA_{(i,\cdot)}= - \bar \bA_{(\cdot,i)}= U_i\,.
\]
\end{lemma}
\begin{proof}
For any $i \in [n]$, observe that since the components of $X$ have a density, then $\p\big(X_i(k)=X_j(k)\big| X_i\big)=0$ so that
\[
\E\big[\sign\big(X_i(k)-X_j(k)\big)\big| X_i\big] =\p\big(X_i(k)>X_j(k))\big| X_i\big)-\p\big(X_i(k)<X_j(k))\big| X_i\big)=2F_k(X_i(k))-1\,.
\]
The observation that $\bar \bA_{(i,\cdot)} = -\bar \bA_{(\cdot, i)}$ follows from the fact that $\bA_{(i,j)} = -\bA_{(j,i)}$.
\end{proof}

Using~\eqref{EQ:ANOVA} we obtain the equality:
\begin{equation}
\label{EQ:ANOVA2}
\bA_{(i,j)}\otimes \bA_{(i,j)} = \bM^{(1)}_{(i,j)} + \bM^{(2)}_{(i,j)} + \Big(\bM^{(2)}_{(i,j)}\Big)^\top + \bM^{(3)}_{(i,j)},
\end{equation}
where
\begin{align*}
  \bM^{(1)}_{(i,j)} &:= I_p + \sD_0[ \{\bar \bA_{(i,\cdot)} + \bar \bA_{(\cdot, j)}\} \otimes \{ \bar \bA_{(i,\cdot)} + \bar \bA_{(\cdot, j)}\}] ,\\
  \bM^{(2)}_{(i,j)} &:=  \sD_0[ \bar \bA_{(i,j)} \otimes \{ \bar \bA_{(i,\cdot)} + \bar \bA_{(\cdot, j)}\}],\\
  \bM^{(3)}_{(i,j)} &:=  \sD_0[ \bar \bA_{(i,j)} \otimes \bar \bA_{(i,j)}].
\end{align*}
By the relation $\bar \bA_{(i,\cdot)} = - \bar \bA_{(\cdot,i)}$ from Lemma \ref{lem:Abar} we have 
\[
\sum_{1 \leq i < j \leq n} \{\bar \bA_{(i,\cdot)} + \bar \bA_{(\cdot, j)}\}\otimes \{\bar \bA_{(i,\cdot)} + \bar \bA_{(\cdot, j)}\} 
= (n-1)\sum_{i = 1}^n \bar \bA_{(i,\cdot)}\otimes \bar \bA_{(i,\cdot)}  -   \sum_{(i,j)\in [n]^2 : i \neq j}^n \bar \bA_{(i,\cdot)} \otimes  \bar \bA_{(j,\cdot)}\,.
\]
Using Lemma \ref{lem:Abar} yields:
\begin{equation}
\label{PR:lem:tausum2}
\frac{1}{\binom{n}{2}} \sum_{1 \le i < j \le n} \bM^{(1)}_{(i,j)} = I_p + \frac{2}{n} \sum_{i=1}^n \sD_0[U_i\otimes U_i]
-  \frac{1}{\binom{n}{2}} \sD_0\Bigg[\sum_{(i,j)\in [n]^2: i\neq j} U_i \otimes U_j\Bigg]\,.
\end{equation}

Next, note that, the coordinates of each $U_i, i=1, \ldots, n$ are mutually independent so that $\E[U_i]=0$ and
\begin{equation}
\label{EQ:expA}
\E\big[U_i \otimes  U_i\big]=\E[T^2]I_p=\frac{1}{3}I_p\,,
\end{equation}
where $T \sim \mathsf{Unif}([-1,1])$. Theorem \ref{thm:mplaw} implies as $n \to \infty$ and $\frac{p}{n} \to \gamma >0$, the empirical spectral distribution of
\[
\frac{2}{n}\sum_{i=1}^n U_i \otimes U_i
\]
converges in probability to $(2/3)Y$, where $Y$ is distributed according to the standard Mar\v{c}enko-Pastur law with parameter $\gamma$. Moreover, 
\begin{align*}
\frac{1}{p} \E \left\|\frac{2}{n}\sum_{i=1}^n\diag\big(U_i\otimes U_i\big)-\frac23 I_p\right\|_{\mathrm{F}}^2 &= \frac{4}{pn^2}\sum_{k=1}^p \E\Bigg(\sum_{i=1}^n\big\{U_i(k)^2 - \E\big[ U_i(k)^2 \big]\big\}\Bigg)^2 \leq \frac{C}{n} \to 0 \,, \\
\frac{1}{p} \E \left\|\frac{1}{\binom{n}{2}} \sD_0\Bigg[\sum_{(i,j)\in [n]^2 : i \neq j} U_i \otimes U_j\Bigg]\right\|_{\mathrm{F}}^2 &= \frac{1}{p\binom{n}{2}^2} \sum_{(k,l) \in [p]^2 : k \neq l} \E \Bigg\{\sum_{(i,j) \in [n]^2: i\neq j} U_i(k) U_j(l)\Bigg\}^2 \leq \frac{C p}{n^2} \to 0\,, 
\end{align*}
for some constant $C > 0$ independent of $n$. By Lemma \ref{thm:hoffwiel}, the normalized Frobenius norm bounds the L\'{e}vy distance between spectral measures. An application of Lemma \ref{thm:hoffwiel}, together with~\eqref{PR:lem:tausum2}, triangle inequality and the above bounds yields the following result.
\begin{proposition}
\label{prop:MP1}
As $n \to \infty$ and $\frac{p}{n} \to \gamma >0$, the empirical spectral distribution $\tilde \mu_p$ of
\[
\frac{1}{\binom{n}{2}}\sum_{1\le i<j\le n} \bM^{(1)}_{(i,j)}
\]
converges in probability to the law of $\frac{2}{3}Y+\frac{1}{3}$, where $Y$ is distributed according to the standard Mar\v{c}enko-Pastur law with parameter $\gamma$.
\end{proposition}

Let $\hat \mu_p^{\ktau}$ denote the empirical spectral distribution of $\ktau$. Using Lemma \ref{thm:hoffwiel} once more, we show that the L\'evy distance between $\hat \mu_p^{\ktau}$ and $\tilde \mu_p$ converges to zero. This implies Theorem \ref{thm:main} by Proposition \ref{prop:MP1}. To that end, observe that by~\eqref{EQ:ANOVA2} and triangle inequality:
\begin{equation}
\label{EQ:limitfrob}
\frac{1}{p}\E\bigg\|\ktau - \frac{1}{\binom{n}{2}}\sum_{1\le i<j\le n} \bM^{(1)}_{(i,j)}\bigg\|_{\mathrm{F}}^2 \leq \frac{2}{p}\E\bigg\|\frac{1}{\binom{n}{2}}\sum_{1\le i<j\le n} \bM_{(i,j)}^{(2)}\bigg\|_{\mathrm{F}}^2  + \frac{1}{p}\E\bigg\|\frac{1}{\binom{n}{2}}\sum_{1\le i<j\le n} \bM_{(i,j)}^{(3)}\bigg\|_{\mathrm{F}}^2  .
\end{equation}

To show that~\eqref{EQ:limitfrob} goes to zero, we notice that the collection of matrices $\big\{ \bM_{(i,j)}^{(3)}\big\}_{1\le i<j\le n}$ satisfies
\begin{equation}
\label{EQ:ortho}
\E\Tr\Big\{\Big(\bM^{(3)}_{(i,j)}\Big)^\top \bM^{(3)}_{(i',j')}\Big\} =
\left\{\begin{array}{ll}
\E \| \bM_{(i, j)}^{(3)}\|_{\mathrm{F}}^2&\text{for}\ (i,j)=(i', j')\\
0&\text{otherwise}\,.
\end{array}\right.
\end{equation}
To see this, expand
\begin{equation}
\label{EQ:frobcalc}
\begin{split}
\E\Tr\Big\{\Big(\bM^{(3)}_{(i,j)}\Big)^\top \bM^{(3)}_{(i',j')}\Big\} = \sum_{(k,l) \in [p]^2 : k \neq l} \Big(&  \E[ \{\bA_{(i,j)}(k) - U_i(k) + U_j(k)\} \{ \bA_{(i',j')}(k) - U_{i'}(k) + U_{j'}(k)\}] \\
 &\times  \E[ \{\bA_{(i,j)}(l) - U_i(l) + U_j(l)\} \{ \bA_{(i',j')}(l) - U_{i'}(l) + U_{j'}(l)\}] \Big),
\end{split}
\end{equation}
and notice that each expectation is zero unless $(i,j) = (i',j')$ by Tower property and Lemma \ref{lem:Abar}.  Note that when $(i,j) = (i',j')$, the expression \eqref{EQ:frobcalc} is bounded by $Cp^2$ for some $C > 0$. The equation \eqref{EQ:ortho} also holds for the collection of matrices $\big\{ \bM^{(2)}_{(i,j)} \big\}_{1 \le i < j\le n}$ and we also have $\E \|\bM^{(2)}_{(i,j)}\|^2 \le Cp^2$ by a similar argument.
Therefore the right side of~\eqref{EQ:limitfrob} is bounded by:
\[ 
\frac{Cp}{\binom{n}{2}^2} \times \card\big\{(i,j,i',j') \in [n]^4\,:\, (i,j)=(i', j')\big\}\le \frac{Cp}{n^2},
\]
for some constant $C > 0$, which vanishes as $n \to \infty$. This concludes the proof of Theorem \ref{thm:main}.

\bigskip \textsc {Acknowledgements:} The authors would like to thank Alice Guionnet for helpful comments and discussions. A.~S.~B. was supported by Grant NSF-DMS-1541100. Part of this work was done while A.~S.~B.~ was with the Department of Mathematics at MIT. A.~L. was supported by Grant NSF-MS-1307704. P.~R was  supported by Grants NSF-DMS-1541099, NSF-DMS-1541100 and DARPA-BAA-16-46 and a grant from the MIT NEC Corporation.

\appendix
\section{The standard Mar\v{c}enko-Pastur law}
We include here the definition of the  standard Mar\v{c}enko-Pastur law and a bound on the distance between empirical spectral distributions of two matrices.

\begin{theorem}[Mar\v{c}enko-Pastur law {\cite[Theorem 3.6]{BaiSil10}}]
\label{thm:mplaw}
Let $X_1, \ldots, X_n$ be independent copies of a random vector $X \in \R^p$ such that\begin{equation*}
\E [X] = 0, \qquad \E [X\otimes X] = I_p\,.
\end{equation*}
Suppose that $n \to \infty, \frac{p}{n} \to \gamma >0$ and define $a = (1 - \sqrt{\gamma})^2$,
and $b = (1 + \sqrt{\gamma})^2$.  Then the empirical spectral distribution of the
matrix
\begin{equation*}
\frac{1}{n} \sum_{i=1}^n X_i \otimes X_i,
\end{equation*}
converges almost surely to the standard Mar\v{c}enko-Pastur law which has density:
\begin{equation*}
p_\gamma(x) = \begin{cases}
\displaystyle \frac{1}{2 \pi x\gamma} \sqrt{(b - x) (x - a)}, \qquad &\text{if } a \leq x \leq b,\\
 0, \qquad &\text{otherwise},
\end{cases}
\end{equation*}
and has a point mass $1 - \frac{1}{\gamma}$ at the origin if $\gamma > 1$. \end{theorem}

\begin{lemma}[{\cite[Corollary A.41]{BaiSil10}}]
\label{thm:hoffwiel} Let $A$ and $B$ be two $p\times p$ normal
matrices, with empirical spectral distributions $\hat \mu^A$ and $\hat
\mu^B$. Then
\[ 
L(\hat \mu^A,\hat \mu^B)^3 \leq \frac{1}{p} \|A
-B\|_{\mathrm{F}}^2,
\] 
where $L(\hat \mu^A,\hat \mu^B)$ is the L\'{e}vy distance between
the distribution functions $\hat \mu^A$ and $\hat \mu^B$.
\end{lemma}

\bibliographystyle{alphaabbr}
\bibliography{kendall}

\end{document}